\documentclass{article}
\usepackage[utf8]{inputenc} % proper encryption
\usepackage[USenglish]{babel} % language 
\usepackage{graphicx} % images
\usepackage{comment} % doing long comments
\usepackage{amsmath}
\usepackage{amsfonts}
\usepackage{amssymb}
\usepackage[onehalfspacing]{setspace}
\usepackage{mathtools} % special math symbols
\usepackage{stmaryrd}
\usepackage[nottoc]{tocbibind}
\usepackage{amsthm}
\usepackage{abstract} % abstract
\usepackage{float}
\usepackage[mathscr]{euscript}
\usepackage[singlelinecheck=false,justification=RaggedRight,format=plain]{caption}
\usepackage{subcaption}
\usepackage[inner=2.5cm, outer=2.5cm, top=2.5cm, bottom=2.5cm]{geometry}
\usepackage{mathabx}
\usepackage{tikz-cd}
\usepackage{xcolor}
\usepackage{array}
\usepackage[noadjust]{cite}
\usepackage{bbm}
\usepackage{authblk}
\usepackage[normalem]{ulem}

\usepackage[pdftex,backref=page,hidelinks]{hyperref}

\hypersetup{linktocpage=true,colorlinks=true,citecolor=blue,urlcolor=blue}
\newcommand{\doi}[1]{\href{https://doi.org/#1}{\ttfamily #1}}

\usepackage{orcidlink}

\newcommand{\cg}{\nabla^{g}}
\newcommand{\ch}{\nabla^{h}}

\newcommand{\cn}{\nabla}

\newcommand{\ric}{\operatorname{ric}}

\newcommand{\frg}{\mathfrak{g}}
\newcommand{\Id}{\mathrm{Id}}
\newcommand{\Der}{\mathrm{Der}}
\newcommand{\End}{\mathrm{End}}

\newcommand{\R}{\mathbb{R}}

\renewcommand{\d}{\mathrm{d}}

\DeclarePairedDelimiter{\escal}{\langle}{\rangle}

\makeatletter
\g@addto@macro\bfseries{\boldmath}
\makeatother

\newtheorem{theorem}{Theorem}[section]
\newtheorem{corollary}[theorem]{Corollary}
\newtheorem{lemma}[theorem]{Lemma}
\newtheorem{proposition}[theorem]{Proposition}

\theoremstyle{definition}
\newtheorem{definition}[theorem]{Definition}
\newtheorem{example}[theorem]{Example}
\newtheorem{remark}[theorem]{Remark}

\addto\captionsUSenglish{}

\let\OLDthebibliography\thebibliography
\renewcommand\thebibliography[1]{
	\OLDthebibliography{#1}
	\setlength{\parskip}{0pt}
}

\allowdisplaybreaks

\title{\textbf{Born Lie algebras}}
\author{Alejandro Gil-García \orcidlink{0000-0002-9370-241X} and Paula Naomi Pilatus \orcidlink{0009-0007-5198-6964}}
\affil{\normalsize Beijing Institute of Mathematical Sciences and Applications (BIMSA)\\
        No.\ 544, Hefangkou Village, Huaibei Town, Huairou District, Beijing 101408, China\\
        alejandrogilgarcia@bimsa.cn}
\affil{\normalsize Fachbereich Mathematik\\
	Universit\"at Hamburg\\
	Bundesstra\ss e 55, 20146 Hamburg, Germany\\
        paula.pilatus@uni-hamburg.de}
\date{\today}

\begin{document}

%\nocite{*}

\maketitle

\begin{abstract}

We show that every Born Lie algebra can be obtained by the bicross product construction starting from two pseudo-Riemannian Lie algebras. We then obtain a classification of all Lie algebras up to dimension four and all six-dimensional nilpotent Lie algebras admitting an integrable Born structure. Finally, we study the curvature properties of the pseudo-Riemannian metrics of the integrable Born structures obtained in our classification results.\bigskip

\emph{Keywords: Born, pseudo-Kähler, bi-Lagrangian, complex product, bicross product}\medskip

\emph{MSC classification: Primary 53C15; Secondary 53C50, 53C55, 22E25}

\end{abstract}

\clearpage

\tableofcontents

%%%%%%%%%%%%%%%%%%%%%%%%%%%%%%%%%%%%%%%%

\section{Introduction}

Born structures were introduced first in~\cite{FREIDEL2014302} in the context of T-duality in string theory and have been studied further e.g.~in~\cite{freidel2014quantum,freidel2015metastring,freidel2017generalised,freidel2019unique,marotta2018parahermitian,marotta2021born}, along with their application in high-energy physics. There are several equivalent definitions of a Born structure that can be found in the literature. The definition of a Born structure that is stated e.g.~in~\cite{freidel2019unique} is in terms of para-Hermitian and para-quaternionic structures. A definition of Born structures given purely in the language of generalized geometry can be found in~\cite{HMS19}. In this paper, we will work with the definition of a Born structure proposed in~\cite{HKP24}, as a diagram
    \begin{center}
        \begin{tikzcd}[column sep=small]& g\arrow[dl,"A"']\arrow{dr}{B} & \\\omega \arrow{rr}{-J}   & &h \ \ ,
     \end{tikzcd}
    \end{center}
where $g$ and $h$ are pseudo-Riemannian metrics and $\omega$ is a non-degenerate two-form, such that the endomorphisms $A,B,J$ satisfy $A^2=B^2=-J^2=\operatorname{Id}$. It is proved in~\cite{HKP24} that this definition gives indeed rise to the same structure as originally introduced in~\cite{FREIDEL2014302}.\medskip 

We will furthermore adopt the notion of \emph{integrability} of a Born structure as defined in~\cite{HKP24}, i.e.~we will say that the Born structure is integrable if the two-form $\omega$ is closed and the endomorphisms $A$, $B$ and $J$ are integrable. While the question of integrability is usually not addressed in the literature about Born structures from high-energy physics, it is geometrically interesting. An integrable Born structure gives rise to a (pseudo-)Kähler structure, a bi-Lagrangian structure as well as a complex product structure on the underlying manifold satisfying certain compatibility conditions.\medskip

Being a relatively new geometric structure, so far not many examples of manifolds admitting (integrable) Born structures have been discussed in the literature. A particularly feasible class of examples are left-invariant Born structures on Lie groups, since in this setting it suffices to work at the level of the Lie algebra. First examples of Born structures on Lie algebras were studied in~\cite{marotta2018parahermitian} and~\cite{marotta2021born}. Furthermore, in~\cite{HKP24} examples of \emph{integrable} Born structures on nilpotent Lie algebras were provided. In this paper we will discuss constructions in order to obtain Born Lie algebras in a more systematic manner, leading to a classification of all Lie algebras up to dimension four and all six-dimensional nilpotent Lie algebras admitting integrable Born structures. All of these nilpotent Lie algebras have a rational basis, hence the corresponding Lie groups admit cocompact lattices~\cite{Mal49}. Since the left-invariant Born structures descend to the nilmanifolds, this gives rise to many examples of compact Born manifolds.

\subsubsection*{Main results and outline}

In this paper we focus on the construction and classification of low-dimensional Lie algebras admitting an integrable Born structure. In Section~\ref{sec:preliminaries} we recall the notion and basic properties of Born Lie algebras. In Section~\ref{sec:bicross_products} we make use of the bicross product construction (see Definition~\ref{def: bicross product}) to show that given two pseudo-Riemannian Lie algebras $(\frg_\pm,h_\pm)$, a linear isometry $Q\colon\frg_+\longrightarrow\frg_-$ and two representations $\varphi$ and $\rho$ satisfying suitable conditions we obtain a Born Lie algebra, see Theorem~\ref{thm:bicross_Born}. Furthermore, every Born Lie algebra can be obtained in this way, see Proposition~\ref{prop:every_Born_is_bicross}. Setting one of the representations to zero, the bicross product reduces to the semidirect product and we are able to construct many examples of Born Lie algebras starting just from flat pseudo-Riemannian Lie algebras, see Corollary~\ref{cor:Born_from_flat}. By a similar strategy we obtain examples of pseudo-hyperkähler and hypersymplectic Lie algebras, see Propositions~\ref{prop: hyperkähler from flat Kähler} and~\ref{prop: HS from flat biLag}, respectively.\medskip

In Section~\ref{sec:classification_low_dim} we classify low-dimensional Lie algebras admitting an integrable Born structure. The classification in dimension two is trivial and the classification in dimension four is obtained in Theorem~\ref{thm:classification_Born_dim4} by comparing the classifications of four-dimensional pseudo-Kähler and bi-Lagrangian Lie algebras and providing an explicit integrable Born structure for each Lie algebra in the theorem. In Theorem~\ref{thm: 6d class nilpotent} we classify all six-dimensional nilpotent Lie algebras admitting an integrable Born structure. To obtain such classification we make use of several results that restrict the possible candidates. First we show in Lemma~\ref{lemma:heis3-to-R3} that it is enough to consider the bicross products $\R^3\bowtie\mathfrak{heis}_3$, where $\mathfrak{heis}_3$ is the three-dimensional Heisenberg Lie algebra, and $\R^3\bowtie\R^3$. In the $\R^3\bowtie\mathfrak{heis}_3$ case we only need to consider flat metrics by Lemma~\ref{lemma:heis_only_flat} and we obtain all such Lie algebras in Proposition~\ref{prop:Born_R3+heis3}. In the $\R^3\bowtie\R^3$ case we first find necessary conditions to admit an integrable Born structure in Lemma~\ref{lemma:center>=3} and then we find such Lie algebras in Proposition~\ref{prop:Born_R3+R3}. The combination of these last two propositions gives us the complete classification of six-dimensional nilpotent Lie algebras admitting an integrable Born structure. Finally, we study curvature properties of the Born Lie algebras obtained in the classifications.

\subsubsection*{Related results}

As explained, Born geometry lies in the intersection of pseudo-Kähler, bi-Lagrangian and complex product geometry. These three geometries are well studied in the context of left-invariant structures on Lie groups or, equivalently, on Lie algebras. In particular, several classification results are known for each of these structures on low-dimensional Lie algebras.\medskip

Pseudo-Kähler Lie algebras have been completely classified in dimension four \cite{Pseudo-Kahler_4dim_06}. In this dimension it is known that if a Lie algebra admits a symplectic structure then it must be solvable \cite{Chu74}. Moreover, some of these pseudo-Kähler Lie algebras are actually Kähler, i.e.~positive-definite. This can only happen for non-nilpotent or abelian Lie algebras due to \cite{Hasegawa89} (see also \cite{DM96}). Six-dimensional nilpotent pseudo-Kähler Lie algebras have been classified in \cite{PK_6dim_nil_04}, and in \cite{LU25} a partial classification is obtained in dimension eight. Further examples of pseudo-Kähler Lie algebras have been constructed using semidirect products \cite{CCO15,ContiGil_23,Conti_Rossi_SegnanDalmasso_2023}. Bi-Lagrangian structures, also known as para-Kähler \cite{CFG96} or Künneth structures \cite{KünnethGeometry23}, have been classified in dimension four \cite{calvaruso_para-Kahler_15} and in dimension six in the nilpotent case \cite{Bi-Lagrangian_Ham_19}.\medskip

Complex product structures on Lie algebras were introduced in \cite{AS05}, where many four-dimensional examples were obtained. In contrast to the pseudo-Kähler and bi-Lagrangian case, there exist four-dimensional non-solvable complex product Lie algebras. The classification in the solvable case was obtained in \cite{Product_4dim_05} and the complete classification in \cite{para-HC_dim4_10}, where complex product structures appear under the name para-hypercomplex. Six-dimensional nilpotent complex product Lie algebras were classified in \cite{complex_product_dim6_08}.

\subsubsection*{Acknowledgements}

We are grateful to D.\ Conti and D.\ Kotschick for their valuable comments. We also thank the referee for the comments on the paper. The work of AGG is supported by the Beijing Institute of Mathematical Sciences and Applications (BIMSA) and partially funded by the Deutsche Forschungsgemeinschaft (DFG, German Research Foundation) under Germany’s Excellence Strategy -- EXC 2121 “Quantum Universe” -- 390833306 and the Deutsche Forschungsgemeinschaft -- SFB-Geschäftszeichen 1624 -- Projektnummer 506632645. The work of PNP is funded by the Deutsche Forschungsgemeinschaft under Germany’s Excellence Strategy -- EXC 2121 “Quantum Universe” -- 390833306.

%%%%%%%%%%%%%%%%%%%%%%%%%%%%%%%%%%%%%%%%

\section{Preliminaries on Born structures}\label{sec:preliminaries}

Let $a$ and $b$ be non-degenerate bilinear forms on the tangent bundle of a smooth manifold $M$. Then the \emph{recursion operator} from $a$ to $b$ is the unique field of invertible endomorphisms $A$, such that $a(A\cdot,\cdot)=b(\cdot,\cdot)$. We will denote this as $a\overset{A}{\longrightarrow} b$. For a detailed discussion of recursion operators see \cite{BK08}.

\begin{definition}[\cite{HKP24}]
\label{def: Born structure}
Let $M$ be a smooth manifold. A \emph{Born structure} on $M$ is a triple $(g,h,\omega)$, where $g$ and $h$ are pseudo-Riemannian metrics and $\omega$ is a non-degenerate two-form such that the recursion operators
    \begin{center}
        \begin{tikzcd}[column sep=small]& g\arrow[dl,"A"']\arrow{dr}{B} & \\\omega \arrow{rr}{-J}   & &h
     \end{tikzcd}
    \end{center}
satisfy $-J^2=A^2=B^2=\operatorname{Id}$.
\end{definition}

\begin{proposition}[\cite{HKP24}]
\label{prop: algebraic properties}
    The recursion operators $A$, $B$ and $J$ pairwise anti-commute and $ABJ=\operatorname{Id}$. In particular $J$ interchanges the $\pm1$-eigenbundles of $A$ and $B$. Moreover, the bilinear forms $g$, $h$ and $\omega$ are, up to sign, invariant under $A$, $B$ and $J$. The precise relations are shown in Table~\ref{tab:comp_bil_forms}.
\end{proposition}

\begin{table}[H]
    \centering
\begin{tabular}{ c|c|c } 
 $g(AX,AY)=-g(X,Y)$ & $h(AX,AY)=h(X,Y)$ & $\omega(AX,AY)=-\omega(X,Y)$\\
  $g(AX,Y)=-g(X,AY)$ & $h(AX,Y)=h(X,AY)$ & $\omega(AX,Y)=-\omega(X,AY)$\\
  &  &  \\
  $g(BX,BY)=g(X,Y)$ & $h(BX,BY)=h(X,Y)$ & $\omega(BX,BY)=-\omega(X,Y)$\\
  $g(BX,Y)=g(X,BY)$ & $h(BX,Y)=h(X,BY)$ & $\omega(BX,Y)=-\omega(X,BY)$\\
   & & \\
$g(JX,JY)=-g(X,Y)$ & $h(JX,JY)=h(X,Y)$ & $\omega(JX,JY)=\omega(X,Y)$ \\ 
 $g(JX,Y)=g(X,JY)$ & $h(JX,Y)=-h(X,JY)$ & $\omega(JX,Y)=-\omega(X,JY)$ \\ 
\end{tabular}
\caption{Transformation properties of $g$, $h$ and $\omega$ under $A,B$ and $J$ (see \cite{HKP24}).}
\label{tab:comp_bil_forms}
\end{table}

We denote by $L_{\pm}$ the $\pm1$-eigenbundles of $A$. From Proposition~\ref{prop: algebraic properties} we obtain:

\begin{corollary}[\cite{HKP24}]
\label{cor: L pm Lagrangian}
    The distributions $L_{\pm}$ are Lagrangian for $\omega$, null for $g$ and orthogonal for $h$. In particular, $g$ has neutral signature. The signature of $h$ is of the form $(2p,2q)$.
\end{corollary}

\begin{definition}[\cite{HKP24}]
    \label{def: int born}
    A Born structure $(g,h,\omega)$ is \emph{integrable} if the two-form $\omega$ is symplectic and two out of the three recursion operators are integrable. A manifold with an integrable Born structure is called \emph{Born}.
\end{definition}

Here, it is enough to require the integrability of only two out of the three recursion operators, due to the following proposition:

\begin{proposition}[\cite{HKP24,IVANOV2005205}]
\label{prop: rec op}
    If two out of the three recursion operators in a Born structure are integrable, so is the third one.
\end{proposition}

Note that if a Born structure is integrable, then $(h,J)$ is a pseudo-Kähler structure. The structure $(\omega,L_+,L_-)$ is what we will call a \emph{bi-Lagrangian structure}:

\begin{definition}
\label{def: bi-lagrangian}
    An \emph{almost bi-Lagrangian structure} is an almost symplectic form $\omega$ together with a pair of complementary Lagrangian distributions~$L_{\pm}$. The almost bi-Lagrangian structure $(\omega,L_+,L_-)$ is called \emph{integrable} or \emph{bi-Lagrangian structure} if~$\omega$ is closed and the distributions~$L_{\pm}$ are integrable to Lagrangian foliations~$\mathcal{L}_{\pm}$.
\end{definition}

\begin{remark}
    Bi-Lagrangian structures appear in the literature also under the name \emph{Künneth structures}~\cite{KünnethGeometry23}. Furthermore, a bi-Lagrangian structure is equivalent to a \emph{para-Kähler structure}, i.e.~a pair $(g,A)$ of a neutral metric $g$ and an almost product structure $A$ such that $g(A\cdot,A\cdot)=-g(\cdot,\cdot)$ and such that $A$ is parallel with respect to the Levi-Civita connection of $g$.
\end{remark}

In particular, integrability of the Born structure implies that $\cg\omega=0$, $\ch\omega=0$ and $\ch J=0$, where $\cg$ and $\ch$ denote the Levi-Civita connections of $g$ and $h$, respectively. Furthermore, given an integrable Born structure the recursion operators $(J,A)$ form a \emph{complex product structure}, i.e.~a pair of a complex and a product structure that anti-commute.\medskip

In this project we will study constructions in order to obtain examples of left-invariant Born structures on Lie groups. For those, it suffices to work at the level of the Lie algebra. Born structures on Lie algebras were already introduced in~\cite{HKP24}. The definition of a Born structure on a Lie algebra is transcribed from Definition~\ref{def: Born structure} in the obvious way:

\begin{definition}[\cite{HKP24}]
\label{def: Born Lie algebras}
Let $\frg$ be a Lie algebra. A \emph{Born structure} on $\frg$ is a triple $(g,h,\omega)$, where $g$ and $h$ are pseudo-Riemannian metrics and $\omega$ is a non-degenerate two-form such that the recursion operators
    \begin{center}
        \begin{tikzcd}[column sep=small]& g\arrow[dl,"A"']\arrow{dr}{B} & \\\omega \arrow{rr}{-J}   & &h
     \end{tikzcd}
    \end{center}
satisfy $-J^2=A^2=B^2=\operatorname{Id}$. The Born structure is called \emph{integrable} if its two-form is closed under the Chevalley-Eilenberg differential, the endomorphism $J$ has vanishing Nijenhuis tensor and the eigenspaces of $A$ are subalgebras of $\frg$. A Lie algebra with an integrable Born structure is called \emph{Born}.
\end{definition}

Clearly, Proposition~\ref{prop: algebraic properties} carries over to the Lie algebra setting. We will denote the eigenspaces of $A$ by $\frg_{\pm}$. For these, an analogous statement to Corollary~\ref{cor: L pm Lagrangian} holds. Given a Lie group $G$, an integrable Born structure on its Lie algebra in the sense of Definition~\ref{def: Born Lie algebras} gives rise to a left-invariant integrable Born structure on $G$ in the sense of Definition~\ref{def: int born}.\medskip

We furthermore introduce the following notion of equivalence of Born structures on Lie algebras:
\begin{definition}
    Two Born structures $(g,h,\omega)$ and $(g',h',\omega')$ on Lie algebras $\frg$ and $\frg'$ are \emph{equivalent}, if there is an isomorphism of Lie algebras $\Phi\colon \frg\longrightarrow\frg'$ such that
    \begin{align*}
        \Phi^*h'=h,\quad \Phi^*g'=g,\quad \Phi^*\omega'=\omega.
    \end{align*}
\end{definition}

Sometimes, it will be more useful to work with the following equivalent integrability condition for a Born structure on a Lie algebra:

\begin{proposition}
    \label{prop: integrability with nabla h}
    A Born structure on a Lie algebra $\frg$ is integrable if and only if the eigenspaces of $A$ are subalgebras of $\frg$ and $J$ is parallel with respect to the Levi-Civita connection of $h$.
\end{proposition}

In particular, this yields the following equivalent description of an integrable Born structure on a Lie algebra:

\begin{corollary}
\label{cor:integrable_Born_Hermitian_version}
    An integrable Born structure on a Lie algebra $\frg$ is equivalent to a pseudo-Kähler structure $(h,J)$ on $\frg$ together with a pair of complementary subalgebras $\frg_{\pm}$ that are orthogonal with respect to $h$ and interchanged by $J$.
\end{corollary}

In view of Corollary~\ref{cor:integrable_Born_Hermitian_version} we will sometimes denote a Born structure by the tuple $(h,J,\frg_{\pm})$ rather than the three bilinear forms $(g,h,\omega)$.\medskip

A first example of a Born Lie algebra is the following:

\begin{example}
\label{ex: abelian is born}
    Every even-dimensional abelian Lie algebra is Born. An integrable Born structure can be constructed from the standard Kähler structure with a canonical choice of subalgebras $\frg_{\pm}$ (see~\cite{HKP24}).
\end{example}

%%%%%%%%%%%%%%%%%%%%%%%%%%%%%%%%%%%%%%%%

\section{Born Lie algebras as bicross products}\label{sec:bicross_products}

In this section we show that every Born Lie algebra can be obtained from two pseudo-Riemannian Lie algebras of the same dimension and signature using \emph{bicross products}.

\begin{definition}[\cite{matched-pairs_90,Hypersymplectic_Andrada_06}]
\label{def: bicross product}
    Let $\frg_+,\frg_-$ be real finite-dimensional Lie algebras and let $\varphi\colon\frg_-\longrightarrow\End(\frg_+)$ and $\rho\colon\frg_+\longrightarrow\End(\frg_-)$ be representations satisfying \begin{equation}\label{eq:bicross_condition}
        \begin{aligned}
        \varphi(X_-)[X_+,Y_+]_+ -[\varphi(X_-)X_+,Y_+]_+-[X_+,\varphi(X_-)Y_+]_+ +\varphi(\rho(X_+)X_-)Y_+-\varphi(\rho(Y_+)X_-)X_+&=0,\\
        \rho(X_+)[X_-,Y_-]_- -[\rho(X_+)X_-,Y_-]_- -[X_-,\rho(X_+)Y_-]_- +\rho(\varphi(X_-)X_+)Y_--\rho(\varphi(Y_-)X_+)X_-&=0
    \end{aligned}
    \end{equation} for all $X_+,Y_+\in\frg_+$ and $X_-,Y_-\in\frg_-$, where $[\cdot,\cdot]_+$ and $[\cdot,\cdot]_-$ denote the Lie brackets of $\frg_+$ and $\frg_-$, respectively. The \emph{bicross product} $\frg_\varphi^\rho:=\frg_+\bowtie_\varphi^\rho\frg_-$ is the Lie algebra with underlying vector space $\frg_+\oplus\frg_-$ and Lie bracket given by $$[X_++X_-,Y_++Y_-]:=[X_+,Y_+]_++\varphi(X_-)Y_+-\varphi(Y_-)X_++[X_-,Y_-]_-+\rho(X_+)Y_--\rho(Y_+)X_-$$ for $X_+,Y_+\in\frg_+$ and $X_-,Y_-\in\frg_-$.
\end{definition}

\begin{remark}
    The Lie algebras $\frg_+$ and $\frg_-$ are Lie subalgebras of the bicross product $\frg_\varphi^\rho=\frg_+\bowtie_\varphi^\rho\frg_-$.   
\end{remark}

\begin{proposition}
\label{prop: born on bicross product}
Let $\frg_+$ and $\frg_-$ be real Lie algebras of the same dimension endowed with pseudo-Riemannian metrics $h_+$ and $h_-$ of the same signature. Let furthermore $Q\colon\frg_+\longrightarrow\frg_-$ be a linear isometry from $(\frg_+,h_+)$ to $(\frg_-,h_-)$. Then there is a (possibly non-integrable) Born structure on $\frg:=\frg_+\oplus\frg_-$, with $h:=h_+\oplus h_-$ and $J$ given by \begin{align*}
J:=\begin{pmatrix}
    0 & -Q^{-1}\\
    Q & 0\end{pmatrix}  
\end{align*} in the splitting $\frg=\frg_+\oplus \frg_-$. In particular, given a pair of representations $\varphi\colon\frg_-\longrightarrow\End(\frg_+)$ and $\rho\colon\frg_+\longrightarrow\End(\frg_-)$ satisfying the conditions in Definition~\ref{def: bicross product}, this Born structure gives rise to a (possibly non-integrable) Born structure on the bicross product $\frg_\varphi^\rho:=\frg_+\bowtie_\varphi^\rho\frg_-$. 
\end{proposition}

\begin{proof}
    Let $\frg$, $h$ and $J$ be as defined in Proposition~\ref{prop: born on bicross product}. Then clearly $J$ interchanges $\frg_{\pm}$ and $J^2=-\Id$. Moreover, by the definition of $h$, the subspaces $\frg_{\pm}$ are orthogonal with respect to $h$ and since $Q$ is an isometry, $h$ is invariant under $J$. Hence, $(h,J,\frg_{\pm})$ define a Born structure on $\frg$ and therefore also on $\frg_\varphi^\rho:=\frg_+\bowtie_\varphi^\rho\frg_-$.
\end{proof}

\begin{definition}
\label{def: born structure on bicross products}
    We call the Born structure on $\frg_\varphi^\rho=\frg_+\bowtie_\varphi^\rho\frg_-$ obtained as in Proposition~\ref{prop: born on bicross product} the \emph{Born structure associated with the data} $(\frg_{\pm},h_{\pm},Q,\varphi,\rho)$. 
\end{definition}

\begin{remark}
    Note that given the metric $h_-$ on $\frg_-$ and the isometry $Q\colon\frg_+\longrightarrow\frg_-$, the metric $h_+$ on $\frg_+$ is completely determined by $h_+=Q^*h_-$.
\end{remark}

Next, we study under which conditions on the representations $\varphi$ and $\rho$ the Born structure associated with the data $(\frg_\pm,h_\pm,Q,\varphi,\rho)$ is integrable. For this, we first need to compute the Levi-Civita connection $\nabla^h$ of the metric $h:=h_+\oplus h_-$ on $\frg_{\varphi}^{\rho}$.\medskip

Given a pseudo-Riemannian Lie algebra $(\frg,h)$ and an endomorphism $f\colon\frg\longrightarrow\frg$, we denote by $f^*$ the adjoint operator of $f$, i.e.~$h(f\cdot,\cdot)=h(\cdot,f^*\cdot)$. Moreover, we write $f=f^s+f^a$, where $f^s=\frac{1}{2}(f+f^*)$ denotes the $h$-symmetric part of $f$ and $f^a=\frac{1}{2}(f-f^*)$ denotes the $h$-antisymmetric part of $f$.

\begin{lemma}\label{lemma:LeviCivita}
    Let $X,Y,Z\in\frg_{\varphi}^{\rho}$. Then the Levi-Civita connection $\nabla^h$ of $h$ is given by: \begin{enumerate}
        \itemsep 0em
        \item $h(\nabla^h_{X_+}Y_+,Z)=h_+(\nabla^+_{X_+}Y_+,Z_+)+h_+(\varphi(Z_-)^sX_+,Y_+)$.
        \item $\nabla^h_{X_+}Y_-=-\varphi(Y_-)^sX_++\rho(X_+)^aY_-$.
        \item $\nabla^h_{X_-}Y_+=\varphi(X_-)^aY_+-\rho(Y_+)^sX_-$.
        \item $h(\nabla^h_{X_-}Y_-,Z)=h_-(\nabla^-_{X_-}Y_-,Z_-)+h_-(\rho(Z_+)^sX_-,Y_-)$.
    \end{enumerate}
    Here, the subscripts $\pm$ denote the projections to $\frg_{\pm}$, respectively, and $\nabla^{\pm}$ denote the Levi-Civita connections of $h_{\pm}$ on $\frg_{\pm}$.
\end{lemma}

\begin{proof}
    Recall the Koszul formula for the Levi-Civita connection $\nabla^h$: $$2h(\nabla^h_uv,w)=h([u,v],w)-h(v,[u,w])-h(u,[v,w])$$ for $u,v,w\in\frg_\varphi^\rho$. Let $u=X_+$, $v=Y_+$ and $w=Z=Z_++Z_-$. Then we have: \begin{align*}
        2h(\nabla^h_{X_+}Y_+,Z)&=h([X_+,Y_+],Z)-h(Y_+,[X_+,Z])-h(X_+,[Y_+,Z])\\
        &=h([X_+,Y_+]_+,Z)\\
        &\quad-h(Y_+,[X_+,Z_+]_+-\varphi(Z_-)X_++\rho(X_+)Z_-)\\
        &\quad-h(X_+,[Y_+,Z_+]_+-\varphi(Z_-)Y_++\rho(Y_+)Z_-)\\
        &=h_+([X_+,Y_+]_+,Z_+)-h_+(Y_+,[X_+,Z_+]_+)-h_+(X_+,[Y_+,Z_+])\\
        &\quad+h_+(Y_+,\varphi(Z_-)X_+)+h_+(X_+,\varphi(Z_-)Y_+)\\
        &=2h_+(\nabla^+_{X_+}Y_+,Z_+)+2h_+(\varphi(Z_-)^sX_+,Y_+),
    \end{align*} hence $h(\nabla^h_{X_+}Y_+,Z)=h_+(\nabla^+_{X_+}Y_+,Z_+)+h_+(\varphi(Z_-)^sX_+,Y_+)$. Similarly we compute \begin{align*}
        2h(\nabla^h_{X_+}Y_-,Z)&=h([X_+,Y_-],Z)-h(Y_-,[X_+,Z])-h(X_+,[Y_-,Z])\\
        &=h(-\varphi(Y_-)X_++\rho(X_+)Y_-,Z)\\
        &\quad-h(Y_-,[X_+,Z_+]_+-\varphi(Z_-)X_++\rho(X_+)Z_-)\\
        &\quad-h(X_+,\varphi(Y_-)Z_++[Y_-,Z_-]_--\rho(Z_+)Y_-)\\
        &=-h_+(\varphi(Y_-)X_+,Z_+)+h_-(\rho(X_+)Y_-,Z_-)\\
        &\quad-h_-(Y_-,\rho(X_+)Z_-)-h_+(X_+,\varphi(Y_-)Z_+)\\
        &=-2h(\varphi(Y_-)^sX_+,Z_+)+2h_-(\rho(X_+)^aY_-,Z_-)\\
        &=2h(-\varphi(Y_-)^sX_++\rho(X_+)^aY_-,Z).
    \end{align*} To obtain the formula for $\nabla^h_{X_-}Y_+$ we use the fact that $\nabla^h$ is torsion-free. Finally, the formula for $\nabla^h_{X_-}Y_-$ is obtained by an analogous computation to that of the formula for $\nabla^h_{X_+}Y_+$.
\end{proof}

By Proposition~\ref{prop: integrability with nabla h}, to show that the Born structure associated with the data $(\frg_\pm,h_\pm,Q,\varphi,\rho)$ is integrable it suffices to show under which conditions on $\varphi$ and $\rho$ the complex structure $J$ is $\nabla^h$-parallel. This is the content of the following result. 

\begin{theorem}\label{thm:bicross_Born}
    The Born structure associated with the data $(\frg_\pm,h_\pm,Q,\varphi,\rho)$ is integrable if and only if \begin{enumerate}
        \itemsep 0em
        \item $\varphi(X_-)^a=Q^{-1}\nabla^-_{X_-}Q$ for all $X_-\in\frg_+$.
        \item $\varphi(QX_+)^sY_+=\varphi(QY_+)^sX_+$ for all $X_+,Y_+\in\frg_+$.
        \item $\rho(X_+)^a=Q\nabla^+_{X_+}Q^{-1}$ for all $X_+\in\frg_+$.
        \item $\rho(X_+)^sQY_+=\rho(Y_+)^sQX_+$ for all $X_+,Y_+\in\frg_+$.
    \end{enumerate}
\end{theorem}

\begin{proof}
    As explained above, to show that the Born structure is integrable it suffices to show that $\nabla^hJ=0$. We will make use of Lemma~\ref{lemma:LeviCivita} throughout the proof. We have to consider four different cases: \medskip
    
    \textbullet\, Case $X_+,Y_+\in\frg_+$: We compute $$h((\nabla^h_{X_+}J)Y_+,Z)=h(\nabla^h_{X_+}JY_+,Z)-h(J\nabla^h_{X_+}Y_+,Z),$$ where \begin{align*}
        h(\nabla^h_{X_+}JY_+,Z)&=h(\nabla^h_{X_+}QY_+,Z)=h(-\varphi(QY_+)^sX_++\rho(X_+)^aQY_+,Z)\\
        &=-h_+(\varphi(QY_+)^sX_+,Z_+)+h_-(\rho(X_+)^aQY_+,Z_-)\\
        h(J\nabla^h_{X_+}Y_+,Z)&=-h(\nabla^h_{X_+}Y_+,J(Z))=-h(\nabla^h_{X_+}Y_+,-Q^{-1}Z_-+QZ_+)\\
        &=h_+(\nabla^+_{X_+}Y_+,Q^{-1}Z_-)-h_+(\varphi(QZ_+)^sX_+,Y_+).
    \end{align*} Hence \begin{align*}
        h((\nabla^h_{X_+}J)Y_+,Z)&=-h_+(\nabla^+_{X_+}Y_+,Q^{-1}Z_-)+h_-(\rho(X_+)^aQY_+,Z_-)\\
        &\quad-h_+(X_+,\varphi(QY_+)^sZ_+-\varphi(QZ_+)^sY_+).
    \end{align*} This is zero if and only if $\varphi(QY_+)^sZ_+=\varphi(QZ_+)^sY_+$ for all $Y_+,Z_+\in\frg_+$ and $\rho(X_+)^a=Q\nabla^+_{X_+}Q^{-1}$ for all $X_+\in\frg_+$ since $h_+(\nabla^+_{X_+}Y_+,Q^{-1}Z_-)=h_-(Q\nabla^+_{X_+}Y_+,Z_-)$.\medskip
    
    \textbullet\, Case $X_+\in\frg_+,Y_-\in\frg_-$: We compute $$h((\nabla^h_{X_+}J)Y_-,Z)=h(\nabla^h_{X_+}JY_-,Z)-h(J\nabla^h_{X_+}Y_-,Z),$$ where \begin{align*}
        h(\nabla^h_{X_+}JY_-,Z)&=-h(\nabla^h_{X_+}Q^{-1}Y_-,Z)\\
        &=-h_+(\nabla^+_{X_+}Q^{-1}Y_-,Z_+)-h_+(\varphi(Z_-)^sX_+,Q^{-1}Y_-)\\
        &=-h_+(Q^{-1}\rho(X_+)^aY_-,Z_+)-h_+(X_+,\varphi(Z_-)^sQ^{-1}Y_-),\\
        h(J\nabla^h_{X_+}Y_-,Z)&=h(J(-\varphi(Y_-)^sX_++\rho(X_+)^aY_-),Z)\\
        &=-h_-(Q\varphi(Y_-)^sX_+,Z_-)-h_+(Q^{-1}\rho(X_+)^aY_-,Z_+)\\
        &=-h_+(X_+,\varphi(Y_-)^sQ^{-1}Z_-)-h_+(Q^{-1}\rho(X_+)^aY_-,Z_+).
    \end{align*} Hence $$h((\nabla^h_{X_+}J)Y_-,Z)=-h_+(X_+,\varphi(Z_-)^sQ^{-1}Y_--\varphi(Y_-)^sQ^{-1}Z_-).$$
    
    This is zero if and only if $\varphi(Z_-)^sQ^{-1}Y_-=\varphi(Y_-)^sQ^{-1}Z_-$ for all $Y_-,Z_-\in\frg_-$. Setting $\tilde{Y}_+=Q^{-1}Y_-$ and $\tilde{Z}_+=Q^{-1}Z_-$ we obtain $\varphi(Q\tilde{Z}_+)^s\tilde{Y}_+=\varphi(Q\tilde{Y}_+)^s\tilde{Z}_+$, which is the condition obtained before.\medskip
    
    \textbullet\, Case $X_-\in\frg_-,Y_+\in\frg_+$: We compute $$h((\nabla^h_{X_-}J)Y_+,Z)=h(\nabla^h_{X_-}JY_+,Z)-h(J\nabla^h_{X_-}Y_+,Z),$$ where \begin{align*}
        h(\nabla^h_{X_-}JY_+,Z)&=h(\nabla^h_{X_-}QY_+,Z)\\
        &=h_-(\nabla^-_{X_-}QY_+,Z_-)+h_-(\rho(Z_+)^sX_-,QY_+),\\
        h(J\nabla^h_{X_-}Y_+,Z)&=h(J(\varphi(X_-)^aY_+-\rho(Y_+)^sX_-),Z)\\
        &=h_-(Q\varphi(X_-)^aY_+,Z_-)+h_+(Q^{-1}\rho(Y_+)^sX_-,Z_+).
    \end{align*} This is zero if and only if $\varphi(X_-)^a=Q^{-1}\nabla^-_{X_-}Q$ for all $X_-\in\frg_-$ and $\rho(Z_+)^sQY_+=\rho(Y_+)^sQZ_+$ for all $Y_+,Z_+\in\frg_+$.\medskip
    
    \textbullet\, Case $X_-,Y_-\in\frg_-$: A similar computation as in the previous cases gives us again the conditions $\varphi(X_-)^a=Q^{-1}\nabla^-_{X_-}Q$ for all $X_-\in\frg_-$ and $\rho(Z_+)^sQY_+=\rho(Y_+)^sQZ_+$ for all $Y_+,Z_+\in\frg_+$.
\end{proof}

Conversely, every Born Lie algebra can be obtained by some data $(\frg_\pm,h_\pm,Q,\varphi,\rho)$:

\begin{proposition}\label{prop:every_Born_is_bicross}
    Let $\frg$ be a Born Lie algebra. Then for every integrable Born structure $(h,J,\frg_{\pm})$ on $\frg$ there exist a pair of representations $\varphi\colon\frg_-\longrightarrow\End(\frg_+)$ and $\rho\colon \frg_+\longrightarrow\End(\frg_-)$ and a linear isometry $Q\colon (\frg_+,h_+:=h|_{\frg_+})\longrightarrow (\frg_-,h_-:=h|_{\frg_-})$ such that $\frg\cong \frg_+\bowtie_\varphi^\rho\frg_-$ and $(h,J,\frg_{\pm})$ is equivalent to the Born structure associated with the data $(\frg_\pm,h_\pm,Q,\varphi,\rho)$.
\end{proposition}

\begin{proof}
    Let $(h,J,\frg_{\pm})$ be an integrable Born structure on $\frg$. Then $\frg_+$ and $\frg_-$ are orthogonal with respect to $h$ and $h_+:=h|_{\frg_+}$ and $h_-:=h|_{\frg_-}$ define pseudo-Riemannian metrics on $\frg_+$ and $\frg_-$, respectively. Furthermore, we know that $J$ interchanges $\frg_+$ and $\frg_-$. We set $Q:=J|_{\frg_+}\colon\frg_+\longrightarrow \frg_-$. Then it follows from $h(J\cdot,J\cdot)=h(\cdot ,\cdot)$ that $Q$ is a linear isometry with respect to $h_+$ and $h_-$, and it follows from $J^2=-\Id$ that
    \begin{equation}\label{eq:complex_structure_Q}
        J=\begin{pmatrix}
            0&-Q^{-1}\\
            Q & 0
        \end{pmatrix}.
    \end{equation}
    Since $(h,J,\frg_{\pm})$ is integrable, the subspaces $\frg_{\pm}$ are Lie subalgebras of $\frg$. Moreover, we have $\frg_+\oplus\frg_-=\frg$. We define
    \begin{align*}
        \varphi\colon \frg_-\longrightarrow \End(\frg_+),\quad \rho\colon \frg_+\longrightarrow \End(\frg_-)    
    \end{align*}
    by setting 
    \begin{align*}
        \varphi(X)Y:=[X,Y]_+,\quad \rho(Y)X:=[Y,X]_-    
    \end{align*}
    for $X\in \frg_-$, $Y\in \frg_+$, where $[\cdot,\cdot]$ denotes the Lie bracket of $\frg$. We claim that $\varphi$ and $\rho$ are Lie algebra representations. Indeed, for $X,Y\in\frg_-$, $Z\in\frg_+$, we have
    \begin{align*}
        \varphi\left([X,Y]\right)Z=&\left[[X,Y],Z\right]_+\\
        =&\left(\left[X,[Y,Z]\right]+\left[Y,[Z,X]\right]\right)_+\\
        =&\left[X,[Y,Z]_+\right]_+ +\left[Y,[Z,X]_+\right]_+\\
        =&\left[\varphi(X),\varphi(Y)\right]Z,
    \end{align*}
    where we used the Jacobi identity for the second equality and that $\frg_-$ is a subalgebra for the third equality. This shows that $\varphi$ is a Lie algebra representation. The proof for $\rho$ works analogously. Clearly, we have $\frg=\frg_+\bowtie_\varphi^\rho\frg_-$ as Lie algebras. By construction the Born structures are equivalent.
\end{proof}

Taking one of the representations to be the trivial representation, say $\rho\equiv 0$, the bicross product reduces to the well-known \emph{semidirect product}. Therefore, by Proposition~\ref{prop: born on bicross product}, given two Lie algebras $\frg_+$ and $\frg_-$ with pseudo-Riemannian metrics $h_+$ and $h_-$, a linear isometry $Q\colon (\frg_+,h_+)\longrightarrow(\frg_-,h_-)$ and a representation $\varphi\colon\frg_-\longrightarrow \Der(\frg_+)$, we can construct a Born structure on $\frg_+\rtimes_\varphi\frg_-$, the \emph{Born structure associated with the data} $(\frg_{\pm}, h_{\pm}, Q,\varphi)$. In this case Theorem~\ref{thm:bicross_Born} reduces to the following statement:

\begin{theorem}\label{thm:integrable_Born_semidirect_product}
    The Born structure associated with the data $(\frg_{\pm},h_{\pm},Q,\varphi)$ is integrable if and only if \begin{enumerate}
        \itemsep 0em
        \item $\frg_+$ is abelian.
        \item $\varphi(X_-)^a=Q^{-1}\nabla^-_{X_-}Q$ for all $X_-\in\frg_-$.
        \item $\varphi(QX_+)^sY_+=\varphi(QY_+)^sX_+$ for all $X_+,Y_+\in\frg_+$.
    \end{enumerate}
    Here $\cn^{\pm}$ denotes the Levi-Civita connection of $h_{\pm}$.
\end{theorem}

As an application of Theorem~\ref{thm:integrable_Born_semidirect_product} we show that we can construct an integrable Born structure from any flat pseudo-Riemannian Lie algebra.

\begin{corollary}\label{cor:Born_from_flat}
    Let $(\frg_-,h_-)$ be a flat pseudo-Riemannian Lie algebra of dimension $n$. Let $\frg_+=\R^n$, $h_+=Q^*h_-$, where $Q\colon\frg_+\longrightarrow \frg_-$ is an isomorphism of vector spaces, and define $\varphi\colon\frg_-\longrightarrow\Der(\frg_+)$ by $\varphi(X_-)=Q^{-1}\nabla^-_{X_-}Q$. Then the Born structure on $\frg_+\rtimes_\varphi\frg_-$ associated with the data $(\frg_\pm,h_\pm,Q,\varphi)$ is integrable. Moreover, the metric $h$ is flat.
\end{corollary}

\begin{proof}
    This is a direct consequence of Theorem~\ref{thm:integrable_Born_semidirect_product} by noticing that such $\varphi(X_-)$ is $h_+$-antisymmetric for all $X_-\in\frg_-$ and $\varphi$ is a representation since $\nabla^-$ is flat. The metric $h=h_+\oplus h_-$ is flat by the Azencott-Wilson theorem (see \cite[Proposition~1.13]{ContiGil_23}) since in this case the pseudo-Riemannian Lie algebra $(\frg_+\rtimes_\varphi\frg_-,h)$ is linearly isometric to the direct sum Lie algebra $(\frg_+\oplus\frg_-,h)$, which is flat.
\end{proof}

Using a similar strategy, we may also construct many \emph{pseudo-hyperkähler} Lie algebras:

\begin{definition}
    A \emph{pseudo-hyperkähler structure} $(h,J,I)$ on a Lie algebra $\frg$ is a pair of anti-commuting complex structures $J$ and $I$ that are both Kähler for a pseudo-Riemannian metric $h$.
\end{definition}

In order to construct pseudo-hyperkähler Lie algebras, we start with an even-dimensional flat pseudo-Kähler Lie algebra:

\begin{proposition}
\label{prop: hyperkähler from flat Kähler}
    Let $(\frg_-,h_-,I_-)$ be a flat pseudo-Kähler Lie algebra of dimension $2n$. Let $\frg_+=\R^{2n}$, $h_+=Q^*h_-$, $I_+=-Q^{-1}I_-Q$, where $Q\colon\frg_+\longrightarrow \frg_-$ is an isomorphism of vector spaces, and define the representation $\varphi\colon\frg_-\longrightarrow\Der(\frg_+)$ by $\varphi(X_-)=Q^{-1}\nabla^-_{X_-}Q$. Set
    \begin{equation*}
h=\begin{pmatrix}
    h_+&0\\
    0&h_-
\end{pmatrix},\quad J=\begin{pmatrix}
    0&-Q^{-1}\\
    Q&0
\end{pmatrix},\quad I=\begin{pmatrix}
    I_+&0\\
    0&I_-
\end{pmatrix}.
\end{equation*}
    Then $(h,J,I)$ is a pseudo-hyperkähler structure on $\frg_+\rtimes_\varphi\frg_-$.
\end{proposition}

\begin{proof}
    Let $\frg_{\pm}$, $\varphi$, $h$, $I$ and $J$ as in Proposition~\ref{prop: hyperkähler from flat Kähler}. Then it is easy to check that $(h,J,I)$ is an almost pseudo-hyperhermitian structure on $\frg_+\rtimes_\varphi\frg_-$. We know by Theorem~\ref{thm:integrable_Born_semidirect_product} that with the above choices, $\nabla^hJ=0$. Moreover, by \cite[Theorem~1.3]{ContiGil_23} we know that $\nabla^hI=0$ for such $\varphi$ if and only if $I_+\circ\varphi(X_-)=\varphi(X_-)\circ I_+$. However, this condition is satisfied automatically since $\nabla^-I_-=0$.
\end{proof}

Similarly, we can construct many examples of hypersymplectic Lie algebras:

\begin{definition}
    A \emph{hypersymplectic structure} on a Lie algebra $\frg$ is a triple $(h,J,E)$ such that $(h,J)$ is pseudo-Kähler, $(h,E)$ is bi-Lagrangian and $J$ and $E$ anti-commute.
\end{definition}

\begin{proposition}
\label{prop: HS from flat biLag}
    Let $(\frg_-,h_-,E_-)$ be a flat bi-Lagrangian Lie algebra of dimension $2n$. Let $\frg_+=\R^{2n}$, $h_+=Q^*h_-$, $E_+=-Q^{-1}E_-Q$, where $Q\colon\frg_+\longrightarrow \frg_-$ is an isomorphism of vector spaces, and define the representation $\varphi\colon\frg_-\longrightarrow\Der(\frg_+)$ by $\varphi(X_-)=Q^{-1}\nabla^-_{X_-}Q$. Set \begin{equation*}
h=\begin{pmatrix}
    h_+&0\\
    0&h_-
\end{pmatrix},\quad J=\begin{pmatrix}
    0&-Q^{-1}\\
    Q&0
\end{pmatrix},\quad E=\begin{pmatrix}
    E_+&0\\
    0&E_-
\end{pmatrix}.
\end{equation*}
    Then $(h,J,E)$ is a hypersymplectic structure on $\frg_+\rtimes_\varphi\frg_-$.
\end{proposition}

\begin{remark}
    The definition of hypersymplectic structures that we use here is equivalent to the one presented in~\cite{Hypersymplectic_Andrada_06}, which is in turn equivalent to the definition originally introduced by Hitchin in~\cite{Hit90}. Another equivalent definition of hypersymplectic structures in terms of three symplectic forms and their recursion operators was discussed in~\cite{BK08}.
\end{remark}

%%%%%%%%%%%%%%%%%%%%%%%%%%%%%%%%%%%%%%%%

\section{Born Lie algebras in low dimensions}\label{sec:classification_low_dim}

In this section, we classify all Lie algebras up to dimension four and all nilpotent Lie algebras in dimension six admitting an integrable Born structure.

\subsection{Two-dimensional classification}

The only two-dimensional Lie algebras are $$\R^2=\escal{e_1,e_2}\quad\text{and}\quad\mathfrak{r}_2=\mathfrak{aff}(\R)=\escal{e_1,e_2\mid[e_1,e_2]=e_2}.$$

By Example~\ref{ex: abelian is born}, the abelian Lie algebra $\R^2$ is Born. The Lie algebra $\mathfrak{r}_2$ admits the following Born structure: $$h=e^1\otimes e^1+e^2\otimes e^2,\quad Je_1=e_2,\quad \mathfrak{g}_+=\langle e_1\rangle,\quad \mathfrak{g}_-=\langle e_2\rangle.$$  
Note that in two dimensions, any Born structure is automatically integrable. We conclude that every two-dimensional Lie algebra is Born.

\subsection{Four-dimensional classification}

In this section, we classify all four-dimensional Lie algebras admitting an integrable Born structure. For this, we make use of the existing classifications for pseudo-Kähler and bi-Lagrangian Lie algebras in the literature: The four-dimensional Lie algebras admitting a pseudo-Kähler structure were classified in \cite{Pseudo-Kahler_4dim_06}. A classification of four-dimensional bi-Lagrangian Lie algebras is given in \cite{calvaruso_para-Kahler_15}. Note that integrable Born structures on the nilpotent Lie algebra $\mathfrak{rh}_3$ were already studied in~\cite[Theorem~40]{HKP24} by making use of the hypersymplectic structure on $\mathfrak{rh}_3$.

\begin{theorem}\label{thm:classification_Born_dim4}
    Let $\frg$ be a non-abelian four-dimensional Lie algebra. Then $\frg$ possesses an integrable Born structure $(h,J,\frg_\pm)$ if and only if $\frg$ is isomorphic to one of the following Lie algebras: $$\mathfrak{rh}_3,\,\mathfrak{rr}_{3,0},\,\mathfrak{r}_2\mathfrak{r}_2,\,\mathfrak{r}'_2,\,\mathfrak{r}_{4,-1,-1},\,\mathfrak{d}_{4,1},\,\mathfrak{d}_{4,2},\,\mathfrak{d}_{4,1/2}.$$
\end{theorem}

\begin{proof}
    Comparing the classifications in \cite{Pseudo-Kahler_4dim_06} and \cite{calvaruso_para-Kahler_15}, we find that the non-abelian four-dimensional Lie algebras admitting both pseudo-Kähler and bi-Lagrangian structures are precisely those listed in Theorem~\ref{thm:classification_Born_dim4}. Each of them is in fact Born: In Table~\ref{tab: 4d born} we provide an example of an integrable Born structure for every Lie algebra in the list. We use the convention $e^i\odot e^j:=e^i\otimes e^j+e^j\otimes e^i$.
\end{proof}

\begin{table}[H]
\begin{center}

\begin{tabular}{|m{0.07\linewidth}|m{0.15\linewidth}|m{0.16\linewidth}|m{0.19\linewidth}|m{0.15\linewidth}|}\hline\smallskip
     & \textbf{Bracket} & \textbf{Metric $h$}& \textbf{Complex} & \textbf{Lagrangian}\\ &\textbf{relations}& &\textbf{structure $J$}& \textbf{subalgebras} \\
     \hline\smallskip
    $\mathfrak{rh}_3$ & $[e_1,e_2]=e_3$ & $e^1\odot e^3+e^2\odot e^4$ & $Je_1=e_2$, $Je_3=e_4$ & $\langle e_1,e_3\rangle$, $\langle e_2,e_4\rangle$\\
    \hline
    $\mathfrak{rr}_{3,0}$& $[e_1,e_2]=e_2$& $\sum_{i=1}^4e^i\otimes e^i$ & $Je_1=e_2$, $Je_3=e_4$ & $\langle e_1,e_3\rangle$, $\langle e_2,e_4\rangle$\\
    \hline
    $\mathfrak{r}_2\mathfrak{r}_2$ & $[e_1,e_2]=e_2$,& & & \\ & $[e_3,e_4]=e_4$ & $\sum_{i=1}^4e^i\otimes e^i$ & $Je_1=e_2$, $Je_3=e_4$ & $\langle e_1,e_3\rangle$, $\langle e_2,e_4\rangle$\\
    \hline
    $\mathfrak{r}_2'$ & $[e_1,e_3]=e_3$,& & & \\ & $[e_1,e_4]=e_4$, & & & \\ &$[e_2,e_3]=e_4$, & $e^1\otimes e^1-e^2\otimes e^2$& & \\ &$[e_2,e_4]=-e_3$& $+e^3\otimes e^3-e^4\otimes e^4$ & $Je_1=e_3$, $Je_2=e_4$ & $\langle e_1,e_2\rangle$, $\langle e_3,e_4\rangle$\\
    \hline
    $\mathfrak{r}_{4,-1,-1}$ & $[e_1,e_4]=-e_1$,& & & \\ & $[e_2,e_4]=e_2$,& & & \\ & $[e_3,e_4]=e_3$ & $-e^1\odot e^3-e^2\odot e^4$ & $Je_1=-e_4$, $Je_2=e_3$ & $\langle e_1,e_3\rangle$, $\langle e_2,e_4\rangle$\\
    \hline
    $\mathfrak{d}_{4,1}$ & $[e_1,e_2]=e_3$,& & & \\ & $[e_1,e_4]=-e_1$,& & & \\ & $[e_3,e_4]=-e_3$ & $-e^1\odot e^3+e^2\odot e^4$ & $Je_1=e_4$, $Je_2=e_3$ & $\langle e_1,e_3\rangle$, $\langle e_2,e_4\rangle$\\
    \hline
    $\mathfrak{d}_{4,2}$ & $[e_1,e_2]=e_3$,& & & \\ & $[e_1,e_4]=-2e_1$,& & & \\ & $[e_2,e_4]=e_2$,& $\tfrac{1}{2}e^1\otimes e^1+e^2\otimes e^2$& & \\ & $[e_3,e_4]=-e_3$ & $+e^3\otimes e^3+2e^4\otimes e^4$ & $Je_1=\tfrac{1}{2}e_4$, $Je_2=e_3$ & $\langle e_1,e_3\rangle$, $\langle e_2,e_4\rangle$\\
    \hline
    $\mathfrak{d}_{4,1/2}$ & $[e_1,e_2]=e_3$,& & & \\ & $[e_1,e_4]=-\tfrac{1}{2}e_1$,& & & \\ & $[e_2,e_4]=-\tfrac{1}{2}e_2$,& & & \\ & $[e_3,e_4]=-e_3$ & $\sum_{i=1}^4e^i\otimes e^i$ & $Je_1=e_2$, $Je_3=-e_4$ & $\langle e_1,e_3\rangle$, $\langle e_2,e_4\rangle$\\
    \hline
    \end{tabular}
    \end{center}
    \caption{List of four-dimensional Lie algebras admitting an integrable Born structure.}
    \label{tab: 4d born}
\end{table}

\begin{corollary}
    A four-dimensional Lie algebra admits an integrable Born structure if and only if it admits both pseudo-Kähler and bi-Lagrangian structures.
\end{corollary}

We note that the only unimodular non-abelian four-dimensional Born Lie algebra is $\mathfrak{rh}_3$, which is nilpotent.

\subsection{Six-dimensional nilpotent classification}

In this section we apply the bicross product construction introduced in Section~\ref{sec:bicross_products} in order to obtain a classification of all six-dimensional nilpotent Lie algebras admitting an integrable Born structure. Our aim is to prove the following:

\begin{theorem}
\label{thm: 6d class nilpotent}
    Let $\frg$ be a non-abelian six-dimensional nilpotent Lie algebra. Then $\frg$ possesses an integrable Born structure $(h,J,\frg_\pm)$ if and only if $\frg$ is isomorphic to one of the following Lie algebras:
    $$\begin{aligned}
    \mathfrak{h}_4 &= (0,0,0,0,12,14+23),\\
    \mathfrak{h}_7 &= (0,0,0,12,13,23),\\
    \mathfrak{h}_8 &= (0,0,0,0,0,12),\\
    \mathfrak{h}_9 &= (0,0,0,0,12,14+25),
    \end{aligned}\qquad\begin{aligned}
    \mathfrak{h}_{10} &= (0,0,0,12,13,14),\\
    \mathfrak{h}_{11} &= (0,0,0,12,13,14+23),\\
    \mathfrak{h}_{13}&=(0,0,0,12,13+14,24).
\end{aligned}$$
\end{theorem}

Here we make use of Salamon's notation \cite{Salamon:ComplexStructures} to denote Lie algebras. For instance $\frg=(0,0,0,12,13,14)$ means that there is a basis $\{e_1,\ldots,e_6\}$ of $\frg$ with dual basis $\{e^1,\ldots,e^6\}$ such that $\d e^1=\d e^2=\d e^3=0$, $\d e^4=e^{12}$, $\d e^5=e^{13}$ and $\d e^6=e^{14}$, where $e^{ij}:=e^i\wedge e^j$.

\begin{remark}
    The fact that the Lie algebras $\mathfrak{h}_4$, $\mathfrak{h}_8$ and $\mathfrak{h}_9$ in Theorem~\ref{thm: 6d class nilpotent} admit an integrable Born structure was already proved in \cite{HKP24} by direct computation, that is by providing explicitly such structure for each Lie algebra.
\end{remark}

We start listing the possible candidates for six-dimensional nilpotent Born Lie algebras.

\begin{proposition}[\cite{PK_6dim_nil_04,Bi-Lagrangian_Ham_19,complex_product_dim6_08}]
\label{prop: 6d pK and biL and CP}
    Let $\frg$ be a non-abelian six-dimensional nilpotent Lie algebra. Then $\frg$ admits a pseudo-Kähler, a bi-Lagrangian and a complex product structure if and only if $\frg$ is isomorphic to one of the following:
    $$\begin{aligned}
    \mathfrak{h}_{2}&=(0,0,0,0,12,34), \\
    \mathfrak{h}_{4}&=(0,0,0,0,12,14+23), \\
    \mathfrak{h}_{5}&=(0,0,0,0,13+42,14+23), \\
    \mathfrak{h}_{6}&=(0,0,0,0,12,13), \\
    \mathfrak{h}_{7}&=(0,0,0,12,13,23), \\
    \mathfrak{h}_{8}&=(0,0,0,0,0,12),
    \end{aligned}\qquad\begin{aligned}
    \mathfrak{h}_{9}&=(0,0,0,0,12,14+25),\\
    \mathfrak{h}_{10}&=(0,0,0,12,13,14), \\
    \mathfrak{h}_{11}&=(0,0,0,12,13,14+23), \\
    \mathfrak{h}_{12}&=(0,0,0,12,13,24), \\
    \mathfrak{h}_{13}&=(0,0,0,12,13+14,24), \\
    \mathfrak{h}_{14}&=(0,0,0,12,14,13+42).
    \end{aligned}$$
\end{proposition}

\begin{proof}
   The six-dimensional nilpotent pseudo-Kähler Lie algebras have been classified in~\cite{PK_6dim_nil_04}. Comparing this list with the classification of six-dimensional nilpotent bi-Lagrangian Lie algebras stated in~\cite{Bi-Lagrangian_Ham_19}, we find that every six-dimensional nilpotent pseudo-Kähler Lie algebra admits a bi-Lagrangian structure. Finally, comparing with the classification of six-dimensional nilpotent Lie algebras admitting complex product structures in~\cite{complex_product_dim6_08}, we observe that, except for $\mathfrak{h}_{15}=(0,0,0,12,13+42,14+23)$, every Lie algebra from the list of six-dimensional nilpotent pseudo-Kähler Lie algebras admits a complex product structure.
\end{proof}

For six-dimensional nilpotent Born Lie algebras, the subalgebras $\frg_{\pm}$ of an integrable Born structure must either be abelian $\R^3$ or isomorphic to the three-dimensional Heisenberg Lie algebra $\mathfrak{heis}_3$. The following lemma shows that we can always assume one of the subalgebras to be abelian:

\begin{lemma}
\label{lemma:heis3-to-R3}
    Let $\frg$ be a six-dimensional nilpotent Lie algebra with an integrable Born structure $(h,J,\frg_{\pm})$ with $\frg_+\cong\frg_-\cong\mathfrak{heis}_3$. Then there is another integrable Born structure $(h,J,\tilde{\frg}_{\pm})$ on $\frg$ with the same pseudo-Riemannian metric $h$ and complex structure $J$ and $\tilde{\frg}_+\cong\mathfrak{heis}_3$ but $\tilde{\frg}_-\cong\R^3$. Conversely given an integrable Born structure $(h,J,\frg_{\pm})$ on $\frg$ with $\frg_+\cong\mathfrak{heis}_3$ and $\frg_-\cong\R^3$ there is another integrable Born structure $(h,J,\tilde{\frg}_{\pm})$ on $\frg$ with $\tilde{\frg}_+\cong\tilde{\frg}_-\cong\mathfrak{heis}_3$.
\end{lemma}

\begin{proof}
    It is shown in~\cite[Proposition~5.1]{complex_product_dim6_08} that given a complex product structure $(J,A)$ on a six-dimensional nilpotent Lie algebra such that the subalgebras coming from the eigenbundles of $A$ are both isomorphic to $\mathfrak{heis}_3$, there is a $\theta\in [0,2\pi)$ such that $(J,A_{\theta})$ is another complex product structure with corresponding subalgebras $\frg_+^{\theta}\cong\mathfrak{heis}_3$ and $\frg^{\theta}_-\cong \R^3$, where
    \begin{align*}
        A_{\theta}=\cos\theta\, A+\sin\theta \, JA=\cos\theta\, A+\sin\theta \, B.
    \end{align*}
    
    Suppose now that for a pseudo-Riemannian metric $h$ on $\frg$, $(h,J,\frg_{\pm})$ is an integrable Born structure on $\frg$. Using the relations presented in Table~\ref{tab:comp_bil_forms}, we observe that for every $X,Y\in \frg$ it holds
    \begin{align*}
        h(A_{\theta}X,A_{\theta}Y)=&\sin^2\theta\,h(AX,AY)+\cos^2\theta h(BX,BY)+\sin\theta\cos\theta\,\left(h(AX,BY)+h(BX,AY)\right)\\
        =&(\sin^2\theta+\cos^2\theta)\,h(X,Y)+\sin\theta\cos\theta\,\left(h(X,ABY)+h(X,BAY)\right)\\
        =&h(X,Y)+\sin\theta\cos\theta\,\left(h(X,ABY)+h(X,-ABY)\right)\\
        =&h(X,Y).
    \end{align*}
    Therefore, setting $\tilde{\frg}_{\pm}:=\frg^{\theta}_{\pm}$, $(h,J,\tilde{\frg}_{\pm})$ is an integrable Born structure on $\frg$. The proof of the converse statement works similarly using again~\cite[Proposition~5.1]{complex_product_dim6_08}. 
\end{proof}

Thanks to Lemma~\ref{lemma:heis3-to-R3}, in order to classify all six-dimensional nilpotent Lie algebras admitting an integrable Born structure, it is enough to consider the cases where only one subalgebra is isomorphic to $\mathfrak{heis}_3$ or both subalgebras are abelian. We study each case separately.

\begin{remark}
    Positive-definite Kähler nilpotent Lie algebras are necessarily abelian by \cite{Hasegawa89} (see also \cite{DM96}). Recall that the metric $h$ of the Born structure is given by $h=h_+\oplus h_-$, where $h_\pm$ are pseudo-Riemannian metrics on the Lie algebras $\frg_\pm$. Therefore, in order to construct non-abelian six-dimensional nilpotent Born Lie algebras, we only consider indefinite metrics $h_\pm$ on the three-dimensional Lie algebras $\R^3$ and $\mathfrak{heis}_3$.
\end{remark}

\subsubsection{Case \texorpdfstring{$\R^3\bowtie\mathfrak{heis}_3$}{R3+heis3}}

First we show that the only metric that we can consider on $\mathfrak{heis}_3$ is the flat one.

\begin{lemma}\label{lemma:heis_only_flat}
    Let $h_-$ be a Lorentzian metric on $\frg_-\cong\mathfrak{heis}_3$, $Q\colon\frg_+\cong\R^3\longrightarrow \frg_-$ an isomorphism of vector spaces and set $h_+:=Q^*h_-$. Then there are representations $\varphi\colon\frg_-\longrightarrow\End(\frg_+)$ and $\rho\colon \frg_+\longrightarrow\End(\frg_-)$ satisfying the conditions in Theorem~\ref{thm:bicross_Born} only if $h_-$ is flat.
\end{lemma}

\begin{proof}
    Lorentzian metrics on $\mathfrak{heis}_3=\escal{e_1,e_2,e_3\mid[e_1,e_2]=e_3}$ have been classified in \cite{Rah92}, where it is shown that there are three non-isometric Lorentzian metrics: \begin{align*}
        h_1&=-e^1\otimes e^1+e^2\otimes e^2+e^3\otimes e^3,\\
        h_2&=e^1\otimes e^1+e^2\otimes e^2-e^3\otimes e^3,\\
        h_3&=e^1\otimes e^1-e^2\otimes e^2+e^2\odot e^3,
    \end{align*} where $e^i\odot e^j=e^i\otimes e^j+e^j\otimes e^i$. Among these metrics, the only one that is flat is $h_3$ (see \cite{RR06}). Now let $h_-$ be a Lorentzian metric on $\mathfrak{heis}_3$. We choose a basis $\{f_1,f_2,f_3\}$ of $\frg_+\cong\R^3$ by setting $f_i:=Q^{-1}e_i$. With respect to these choices of bases the complex structure $J$ on $\frg_\varphi^\rho$ defined by \eqref{eq:complex_structure_Q} just takes the form \begin{equation}\label{eq:J_with_Q=identity}
        J=\begin{pmatrix}
        0&-\mathbbm{1}_3\\
        \mathbbm{1}_3&0
    \end{pmatrix},
    \end{equation} where $\mathbbm{1}_3$ denotes the $3\times3$ identity matrix. Suppose that $h_-=h_1$. Then the covariant derivative endomorphisms are given by $$\nabla^-_{e_1}=\begin{pmatrix}
        0 & 0 & 0 \\
        0 & 0 & -\frac{1}{2} \\
        0 & \frac{1}{2} & 0
    \end{pmatrix},\quad\nabla^-_{e_2}=\begin{pmatrix}
        0 & 0 & -\frac{1}{2} \\
        0 & 0 & 0 \\
        -\frac{1}{2} & 0 & 0
    \end{pmatrix},\quad\nabla^-_{e_3}=\begin{pmatrix}
        0 & -\frac{1}{2} & 0 \\
        -\frac{1}{2} & 0 & 0 \\
        0 & 0 & 0
    \end{pmatrix}.$$

    A map $\varphi\colon\frg_-\longrightarrow\End(\frg_+)$ satisfies the conditions in Theorem~\ref{thm:bicross_Born} if and only if is of the form $$\varphi(e_1)=\begin{pmatrix}
        x_{1} & x_{2} & x_{3} \\
        -x_{2} & x_{5} & x_{6} \\
        -x_{3} & x_{6} + 1 & x_{9}
    \end{pmatrix},\quad\varphi(e_2)=\begin{pmatrix}
        x_{2} & -x_{5} & -x_{6} - 1 \\
        x_{5} & y_{5} & y_{6} \\
        x_{6} & y_{6} & z_{6}
    \end{pmatrix},\quad\varphi(e_3)=\begin{pmatrix}
        x_{3} & -x_{6} - 1 & -x_{9} \\
        x_{6} & y_{6} & z_{6} \\
        x_{9} & z_{6} & z_{9}
    \end{pmatrix}.$$
    
    However, it is not hard to see that such map $\varphi\colon\frg_-\longrightarrow\End(\frg_+)$ is never a representation, that is, it never satisfies $[\varphi(e_i),\varphi(e_j)]=\varphi([e_i,e_j])$ for all $i,j$, which in this particular case amounts to $[\varphi(e_1),\varphi(e_2)]=\varphi(e_3)$ and $[\varphi(e_1),\varphi(e_3)]=[\varphi(e_2),\varphi(e_3)]=0$. The same argument applies if we consider the metric $h_2$.
\end{proof}

\begin{proposition}\label{prop:Born_R3+heis3}
    Let $\frg$ be a non-abelian six-dimensional nilpotent Lie algebra. Then it admits an integrable Born structure with $\frg_-\cong\mathfrak{heis}_3$ if and only if it is isomorphic to one of the following:
    \begin{align*}
    \mathfrak{h}_{4}&=(0,0,0,0,12,14+23),\\
        \mathfrak{h}_{7}&=(0,0,0,12,13,23),\\
        \mathfrak{h}_{10}&=(0,0,0,12,13,14),\\
        \mathfrak{h}_{11}&=(0,0,0,12,13,14+23),\\
        \mathfrak{h}_{13}&=(0,0,0,12,13+14,24).
        \end{align*}
\end{proposition}

\begin{proof}
    Let us consider $\frg_-=\mathfrak{heis}_3=\escal{e_1,e_2,e_3\mid[e_1,e_2]=e_3}$ equipped with the flat Lorentzian metric $h_-=e^1\otimes e^1-e^2\otimes e^2+e^2\odot e^3$ and $\frg_+\cong\R^3$. Let $Q\colon\frg_+\longrightarrow\frg_-$ be a vector space isomorphism and set $h_+:=Q^*h_-$. On $\frg_+$ we choose a basis $\{f_1,f_2,f_3\}$ by setting $f_i:=Q^{-1}e_i$. With respect to these choices of bases, the complex structure $J$ on $\frg_\varphi^\rho$ takes the form \eqref{eq:J_with_Q=identity}. By direct computation, we get that two representations $\varphi\colon\frg_-\longrightarrow\End(\frg_+)$ and $\rho\colon\frg_+\longrightarrow\End(\frg_-)$ satisfy \eqref{eq:bicross_condition} and all the conditions in Theorem~\ref{thm:bicross_Born} if and only if they are of the form $$\varphi(e_1)=\begin{pmatrix}
        0&0&0\\
        0&0&0\\
        0&x&0
    \end{pmatrix},\quad\varphi(e_2)=\begin{pmatrix}
        0&x+1&0\\
        0&0&0\\
        x-1&x_0&0
    \end{pmatrix},\quad\varphi(e_3)=0$$ and $$\rho(f_1)=\begin{pmatrix}
        0&0&0\\
        0&0&0\\
        0&y&0
    \end{pmatrix},\quad\rho(f_2)=\begin{pmatrix}
        0&y&0\\
        0&0&0\\
        y&y_0&0
    \end{pmatrix},\quad\rho(f_3)=0,$$ where $x_0,y_0,x,y\in\R$. We check a posteriori that the isomorphism class of $\frg_\varphi^\rho$ do not depend on the particular values of $x_0$ and $y_0$, so we set them to zero now. Comparing some numerical invariants using Table~\ref{tab:numerical_invariants} we obtain the following: $$\frg_\varphi^\rho\cong\begin{cases}
        \mathfrak{h}_{4}=(0,0,0,0,12,14+23)&\mbox{if }y=0,x=-1,\\
        \mathfrak{h}_{10}=(0,0,0,12,13,14)&\mbox{if }y=0,x=0,\\
        \mathfrak{h}_{7}=(0,0,0,12,13,23)&\mbox{if }y=0,x=1,\\
        \mathfrak{h}_{11}=(0,0,0,12,13,14+23)&\mbox{if }y=0,x\not\in\{-1,0,1\},\\
        \mathfrak{h}_{13}=(0,0,0,12,13+14,24)&\mbox{otherwise}.
    \end{cases}$$
\end{proof}

\subsubsection{Case \texorpdfstring{$\R^3\bowtie\R^3$}{R3+R3}}

In this case we can find necessary conditions for the existence of an integrable Born structure.

\begin{lemma}\label{lemma:center>=3}
    Suppose that $\frg$ is a six-dimensional two-step nilpotent Lie algebra with an integrable Born structure $(g,h,\omega)$ such that both Lagrangian subalgebras $\frg_{\pm}$ are abelian. Then the center of $\frg$ is at least three-dimensional.
\end{lemma}

\begin{proof}
    Let $(g,h,\omega)$ be an integrable Born structure on a six-dimensional Lie algebra $\frg$ with recursion operators $A$, $B$ and $J$ as in Definition~\ref{def: Born Lie algebras}. Suppose furthermore that the eigenspaces $\frg_{\pm}$ of $A$ are both abelian subalgebras of $\frg$. By \cite[Theorem~4.1 and Lemma~4.2]{complex_product_dim6_08} we know that there is a basis $\{f_1,f_2,f_3\}$ of $\frg_+$ and a basis $\{e_1,e_2,e_3\}$ of $\frg_-$ such that $Je_i=f_i$ and $\escal{e_3,f_3}\subset\mathfrak{z}$, where $\mathfrak{z}$ denotes the center of $\mathfrak{g}$. We may furthermore assume that $\omega=\sum\omega_{ij}\,e^i\wedge f^j$ and the condition $J^*\omega=\omega$ implies that $\omega_{ij}=\omega_{ji}$.\medskip
    
    Suppose that $\dim\mathfrak{z}=2$, i.e.~$\mathfrak{z}=\escal{e_3,f_3}$. Then, comparing with the list in Proposition~\ref{prop: 6d pK and biL and CP}, we also have $[\frg,\frg]=\escal{e_3,f_3}$. Since $\omega$ is closed and $e_3, f_3\in\mathfrak{z}$ we find
    \begin{align*}
        0=\d\omega(X,Y,e_3)=\omega([X,Y],e_3)\quad \forall X,Y\in\mathfrak{g}.
    \end{align*}
    
    Similarly, it follows that $\omega([X,Y],f_3)=0$ for every $X,Y\in\frg$. Since $[\frg,\frg]=\escal{e_3,f_3}$, this implies that $\omega_{33}=0$. Furthermore, using that $\frg_{\pm}$ are both abelian, we have
    \begin{align*}
        0=&\d\omega(e_1,e_2,f_1)=\omega([f_1,e_1],e_2)+\omega([e_2,f_1],e_1),\\
        0=&\d\omega(e_1,e_2,f_2)=\omega([f_2,e_1],e_2)+\omega([f_2,e_1],f_1),\\
        0=&\d\omega(e_1,f_1,f_2)=\omega([e_1,f_1],f_2)+\omega([f_2,e_1],f_1),\\
        0=&\d\omega(e_2,f_1,f_2)=\omega([e_2,f_1],f_2)+\omega([f_2,e_2],f_1).
    \end{align*}

    By the integrability of $J$ we know that
    \begin{align*}
        0=N_J(e_1,e_2)=-J([f_1,e_2]+[e_1,f_2])
    \end{align*}
    and, hence, $[e_1,f_2]=[e_2,f_1]$. We write
    \begin{align*}
        [e_1,f_1]=A_{11}\,e_3+B_{11}\,f_3,\quad [e_1,f_2]=[e_2,f_1]=A_{12}\,e_3+B_{12}\,f_3,\quad [e_2,f_2]=A_{22}\,e_3+B_{22}\,f_3
    \end{align*}
    and find
    \begin{align*}
        B_{11}\omega_{23}-B_{12}\omega_{13}=&0,\\
        B_{12}\omega_{23}-B_{22}\omega_{13}=&0,\\
        A_{11}\omega_{23}-A_{12}\omega_{13}=&0,\\
        A_{12}\omega_{23}-A_{22}\omega_{13}=&0.
    \end{align*}
    Since $\omega$ is non-degenerate and $\omega_{33}=0$, we must have $\omega_{13}\neq0$ or $\omega_{23}\neq 0$. If $\omega_{13}\neq 0$ and $\omega_{23}=0$, it follows that
    \begin{align*}
        B_{12}=0=B_{22}=A_{12}=A_{22}
    \end{align*}
    and therefore
    \begin{align*}
        [e_1,f_2]=0=[e_2,f_2]=[f_1,e_2],
    \end{align*}
    i.e.~$\escal{e_2,f_2}\subset\mathfrak{z}$, a contradiction. Similarly, if $\omega_{13}=0$ and $\omega_{23}\neq0$, it follows that $\escal{e_1,f_1}\subset\mathfrak{z}$. Now suppose that booth $\omega_{13}$ and $\omega_{23}$ are non-zero. Then we may assume that $\omega_{13}=1=\omega_{23}$. This yields
    \begin{align*}
        A_{11}=&A_{12}=A_{22},\\
        B_{11}=&B_{12}=B_{22},
    \end{align*}
    which implies $\escal{e_1-e_2,f_1-f_2}\subset\mathfrak{z}$, which again contradicts $\mathfrak{z}=\escal{e_3,f_3}$.
\end{proof}

\begin{proposition}\label{prop:Born_R3+R3}
    Let $\frg$ be a non-abelian six-dimensional nilpotent Lie algebra. Then $\frg$ admits an integrable Born structure with both subalgebras $\frg_{\pm}$ abelian if and only if either $$\frg\cong\mathfrak{h}_8=(0,0,0,0,0,12)\quad\text{or}\quad\frg\cong\mathfrak{h}_9=(0,0,0,0,12,14+25).$$
\end{proposition}

\begin{proof}
    By \cite[Section~5.3]{complex_product_dim6_08}, the only candidates to admit an integrable Born structure with both $\frg_{\pm}$ abelian are $\mathfrak{h}_2,\mathfrak{h}_4,\mathfrak{h}_5,\mathfrak{h}_8,\mathfrak{h}_9$ (the Lie algebra $(0,0,0,0,0,12+34)$ does not admit pseudo-Kähler structures). The Lie algebras $\mathfrak{h}_2,\mathfrak{h}_4,\mathfrak{h}_5$ are two-step nilpotent and have two-dimensional center, hence they do not admit an integrable Born structure by Lemma~\ref{lemma:center>=3}. Now let $h_-=-e^1\otimes e^1+e^2\otimes e^2+e^3\otimes e^3$ be the (unique up to isometry) Lorentzian metric on $\frg_-=\escal{e_1,e_2,e_3}\cong\R^3$. Let $Q\colon\frg_+\longrightarrow\frg_-$ be a vector space isomorphism and set $h_+:=Q^*h_-$. On $\frg_+\cong\R^3$ we choose a basis $\{f_1,f_2,f_3\}$ by setting $f_i:=Q^{-1}e_i$. With respect to these choices of bases, the complex structure $J$ on $\frg_\varphi^\rho$ takes the form \eqref{eq:J_with_Q=identity}. If we consider $\rho\equiv0$ and the representation $\varphi\colon\frg_-\longrightarrow\End(\frg_+)$ given by $$\varphi(e_1)=\varphi(e_2)=\begin{pmatrix}
        1&1&0\\
        -1&-1&0\\
        0&0&0
    \end{pmatrix}\quad\text{and}\quad\varphi(e_3)=0,$$ all the conditions of Theorem~\ref{thm:bicross_Born} are fulfilled and hence we have an integrable Born structure on $\frg_\varphi^\rho$. One can check that in this case $\frg_\varphi^\rho\cong\mathfrak{h}_8=(0,0,0,0,0,12)$ by comparing numerical invariants using Table~\ref{tab:numerical_invariants}. If we consider $\rho\equiv0$ and the representation $\varphi\colon\frg_-\longrightarrow\End(\frg_+)$ defined by $$\varphi(e_1)=\begin{pmatrix}
        0&0&1\\
        0&0&1\\
        -1&1&0
    \end{pmatrix},\quad\varphi(e_2)=\begin{pmatrix}
        0&0&-1\\
        0&0&-1\\
        1&-1&0
    \end{pmatrix},\quad\varphi(e_3)=\begin{pmatrix}
        1&-1&0\\
        1&-1&0\\
        0&0&0
    \end{pmatrix},$$ then $\frg_\varphi^\rho\cong\mathfrak{h}_9=(0,0,0,0,12,14+25)$.
\end{proof}

To obtain Theorem~\ref{thm: 6d class nilpotent} we just combine Proposition~\ref{prop:Born_R3+heis3} and Proposition~\ref{prop:Born_R3+R3}.

\subsection{Curvature properties}

We now study the curvature properties of the pseudo-Riemannian metrics $h$ and $g$ for the integrable Born structures listed in Table~\ref{tab: 4d born}. We check whether $h$ and $g$ are flat, Einstein or a \emph{Ricci soliton}. Recall that a pseudo-Riemannian Lie algebra $(\frg,h)$ is an \emph{(algebraic) Ricci soliton} if there exist $\lambda\in\R$ and a derivation $D\in\Der(\frg)$ such that $\ric(h)=\lambda\Id+D$. The curvature properties of the pseudo-Kähler metric $h$ were also studied in \cite[Section 4]{Pseudo-Kahler_4dim_06}. For the Lie algebra $\mathfrak{rr}_{3,0}$ with the metric $h$ from Table~\ref{tab: 4d born} we have that $$\ric(h)=\left(\begin{smallmatrix}
    -1 & 0 & 0 & 0 \\
    0 & -1 & 0 & 0 \\
    0 & 0 & 0 & 0 \\
    0 & 0 & 0 & 0
\end{smallmatrix}\right),\quad\lambda=-1,\quad D=\left(\begin{smallmatrix}
    0 & 0 & 0 & 0 \\
    0 & 0 & 0 & 0 \\
    0 & 0 & 1 & 0 \\
    0 & 0 & 0 & 1
\end{smallmatrix}\right).$$

For the Lie algebra $\mathfrak{d}_{4,2}$ with the metric $h$ from Table~\ref{tab: 4d born} we have that $$\ric(h)=\left(\begin{smallmatrix}
    -3 & 0 & 0 & 0 \\
    0 & 0 & 0 & 0 \\
    0 & 0 & 0 & 0 \\
    0 & 0 & 0 & -3
\end{smallmatrix}\right),\quad\lambda=-3,\quad D=\left(\begin{smallmatrix}
    0 & 0 & 0 & 0 \\
    0 & 3 & 0 & 0 \\
    0 & 0 & 3 & 0 \\
    0 & 0 & 0 & 0
\end{smallmatrix}\right).$$
A summary of the curvature properties of $h$ and $g$ for the integrable Born structures listed in Table~\ref{tab: 4d born} is given in Table~\ref{tab: curvature props 4d}.
\begin{table}[H]
\begin{center}
\begin{tabular}{|c|c|c|}
    \hline
    $\frg$&$h$&$g$\\
    \hline
    \hline
    $\mathfrak{rh}_3$&Flat&Flat\\
    \hline
    $\mathfrak{rr}_{3,0}$&Ricci soliton&Flat\\
    \hline
    $\mathfrak{r}_2\mathfrak{r}_2$&Einstein&Flat\\
    \hline
    $\mathfrak{r}'_2$&Einstein&Flat\\
    \hline
    $\mathfrak{r}_{4,-1,-1}$&Flat&Flat\\
    \hline
    $\mathfrak{d}_{4,1}$&Non-flat&Flat\\
    \hline
    $\mathfrak{d}_{4,2}$&Ricci soliton&Flat\\
    \hline
    $\mathfrak{d}_{4,1/2}$&Einstein&Flat\\
    \hline
\end{tabular}
\end{center}
\caption{Curvature properties of the four-dimensional Born Lie algebras listed in Table~\ref{tab: 4d born}. If the metric is neither flat, Einstein or Ricci soliton we just write non-flat.}
\label{tab: curvature props 4d}
\end{table}

\begin{remark}
    If we consider the Lie algebra $\mathfrak{r}_{4,-1,-1}$ equipped with the integrable Born structure $(h,J,\frg_\pm)$ given by $$h=-e^1\otimes e^1-e^1\odot e^3-e^2\odot e^4-e^4\otimes e^4,\quad Je_1=-e_4,Je_2=e_3,\quad \frg_+=\escal{e_1,e_3},\quad\frg_-=\escal{e_2,e_4},$$ then the pseudo-Kähler metric $h$ is Ricci-flat but non-flat. However, the metric $g$ is still flat.
\end{remark}

Finally, we study some curvature properties of the metrics $h$ and $g$ of the six-dimensional nilpotent Born Lie algebras. We first note the following consequence of Theorem~\ref{thm: 6d class nilpotent}.

\begin{corollary}
    Let $\frg$ be a six-dimensional nilpotent Born Lie algebra. Then for every integrable Born structure $(h,J,\frg_{\pm})$ on $\frg$ the pseudo-Riemannian metric $h$ is flat when restricted to the subalgebras $\frg_{\pm}$.
\end{corollary}

\begin{proof}
    We have seen that every six-dimensional nilpotent Born Lie algebra $\frg$ in Theorem~\ref{thm: 6d class nilpotent} is obtained from flat pseudo-Riemannian Lie algebras $\frg_\pm$.
\end{proof}

Every pseudo-Kähler metric on a nilpotent Lie algebra is Ricci-flat by \cite[Lemma 6.4]{FinoPartonSalamon}. Hence we will only consider whether the metrics are flat or not. We summarize this in Table~\ref{tab: curvature props nilpotent 6d} (for the Lie algebras $\mathfrak{h}_8$ and $\mathfrak{h}_9$ compare with \cite[Corollary 4.2]{PK_6dim_nil_04}).

\begin{table}[H]
\begin{center}
\begin{tabular}{|c|c|c|}
    \hline
    $\frg$&$h$&$g$\\
    \hline
    \hline
    $\mathfrak{h}_{4}$&Non-flat&Flat\\
    \hline
    $\mathfrak{h}_{7}$&Non-flat&Flat\\
    \hline
    $\mathfrak{h}_{8}$&Flat&Flat\\
    \hline
    $\mathfrak{h}_{9}$&Non-flat&Flat\\
    \hline
    $\mathfrak{h}_{10}$&Flat&Flat\\
    \hline
    $\mathfrak{h}_{11}$&Non-flat&Flat\\
    &(for $x\neq-3$)&\\
    \hline
    $\mathfrak{h}_{13}$&Non-flat&Non-flat\\
    &(for $y^2+x^2+3x\neq0$)&(for $y(2x+3)\neq0$)\\
    \hline
\end{tabular}
\end{center}
\caption{Curvature properties of the six-dimensional nilpotent Born Lie algebras of Theorem~\ref{thm: 6d class nilpotent}.}
\label{tab: curvature props nilpotent 6d}
\end{table}

We note that, in contrast with the four-dimensional case (see Table~\ref{tab: curvature props 4d}), we have an example of a Born Lie algebra, namely $\mathfrak{h}_{13}$, for which the neutral metric $g$ is non-flat.

%%%%%%%%%%%%%%%%%%%%%%%%%%%%%%%%%%%%%%%%

\appendix

\section{Some numerical invariants}

As we have already explained, a Born Lie algebra is in particular a pseudo-Kähler Lie algebra. Six-dimensional nilpotent pseudo-Kähler Lie algebras were classified in \cite[Theorem 3.1]{PK_6dim_nil_04}. To identify the Lie algebras obtained in Theorem~\ref{thm: 6d class nilpotent} we make use of the following table of numerical invariants. The second column corresponds to the nilpotency step of the nilpotent Lie algebra $\frg$. The third, fourth and fifth columns correspond to the first, second and third Betti number, respectively. The sixth column corresponds to the dimension of the automorphism group. The seventh column corresponds to the dimension of the center of the Lie algebra. Finally, the last column corresponds to the number of linearly independent decomposable exact two-forms. We notice that with this set of numbers we can completely identify these Lie algebras.

\begin{table}[H]
\begin{center}
\begin{tabular}{|l|c|c|c|c|c|c|c|}
     \hline
     $\frg$&step&$b_1(\frg)$&$b_2(\frg)$&$b_3(\frg)$&$\dim\mathrm{Aut}(\frg)$&$\dim\mathfrak{z}(\frg)$&$\nu^2(\frg)$\\
     \hline\hline
     $\mathfrak{h}_{8}=(0,0,0,0,0,12)$&2&5&11&14&24&4&1\\\cline{3-8}
     $\mathfrak{h}_{2}=(0,0,0,0,12,34)$&2&4&8&10&16&2&2\\
     $\mathfrak{h}_{4}=(0,0,0,0,12,14+23)$&2&4&8&10&17&2&1\\
     $\mathfrak{h}_{5}=(0,0,0,0,13+42,14+23)$&2&4&8&10&16&2&0\\
     $\mathfrak{h}_{6}=(0,0,0,0,12,13)$&2&4&9&12&19&3&2\\\cline{3-8}
     $\mathfrak{h}_{7}=(0,0,0,12,13,23)$&2&3&8&12&18&3&3\\
     \hline
     $\mathfrak{h}_{9}=(0,0,0,0,12,14+25)$&3&4&7&8&15&2&1\\\cline{3-8}
     $\mathfrak{h}_{10}=(0,0,0,12,13,14)$&3&3&6&8&15&2&3\\
     $\mathfrak{h}_{11}=(0,0,0,12,13,14+23)$&3&3&6&8&14&2&2\\
     $\mathfrak{h}_{12}=(0,0,0,12,13,24)$&3&3&6&8&13&2&3\\
     $\mathfrak{h}_{13}=(0,0,0,12,13+14,24)$&3&3&5&6&12&2&3\\
     $\mathfrak{h}_{14}=(0,0,0,12,14,13+42)$&3&3&5&6&13&2&2\\
     $\mathfrak{h}_{15}=(0,0,0,12,13+42,14+23)$&3&3&5&6&12&2&1\\
     \hline
\end{tabular}
\caption{Some numerical invariants of six-dimensional nilpotent pseudo-Kähler Lie algebras.}
\label{tab:numerical_invariants}
\end{center}
\end{table}

%%%%%%%%%%%%%%%%%%%%%%%%%%%%%%%%%%%%%%%%

\bibliographystyle{amsplain}
\bibliography{biblio}

\end{document}